\documentclass[12pt]{amsart}
\usepackage{amssymb}
\usepackage{amsmath, amscd}
\usepackage{amsthm}
\usepackage{xcolor}
\usepackage{ulem}
\usepackage{tikz}
\usepackage[colorlinks=true, linkcolor=blue,urlcolor=blue]{hyperref}

\newtheorem{theorem}{Theorem}[section]

\newtheorem{proposition}[theorem]{Proposition}
\newtheorem{lemma}[theorem]{Lemma}

\newtheorem{corollary}[theorem]{Corollary}
\theoremstyle{definition}

\newtheorem{example}[theorem]{Example}
\newtheorem{definition}[theorem]{Definition}

\newtheorem{problem}[theorem]{Problem}

\begin{document}

\author[P. Danchev]{Peter Danchev}
\address{Institute of Mathematics and Informatics, Bulgarian Academy of Sciences, 1113 Sofia, Bulgaria}
\email{danchev@math.bas.bg; pvdanchev@yahoo.com}
\author[A. Javan]{Arash Javan}
\address{Department of Mathematics, Tarbiat Modares University, 14115-111 Tehran Jalal AleAhmad Nasr, Iran}
\email{a.darajavan@modares.ac.ir; a.darajavan@gmail.com}
\author[A. Moussavi]{Ahmad Moussavi}
\address{Department of Mathematics, Tarbiat Modares University, 14115-111 Tehran Jalal AleAhmad Nasr, Iran}
\email{moussavi.a@modares.ac.ir; moussavi.a@gmail.com}

\title{\small Rings with \large$u^n-1$ \small nilpotent for each unit \large$u$}

\maketitle




\begin{abstract}
We continue the study in-depth of the so-called $n$-UU rings for any $n\geq 1$, that were defined by the first-named author in Toyama Math. J. (2017) as those rings $R$ for which $u^n-1$ is always a nilpotent for every unit $u\in R$. Specifically, for any $n\geq 2$, we prove that a ring is strongly $n$-nil-clean if, and only if, it is simultaneously strongly $\pi$-regular and an $(n-1)$-UU ring. This somewhat extends results due to Diesl in J. Algebra (2013), Abyzov in Sib. Math. J. (2019) and Cui-Danchev in J. Algebra Appl. (2020). Moreover, our results somewhat improves the ones obtained by Ko\c{s}an et al. in Hacettepe J. Math. Stat. (2020).
\end{abstract}

\section{Introduction and Motivation}

Throughout the present paper, let $R$ be an associative but {\it not} necessarily commutative ring with identity element often stated as $1$. As usual, for such a ring $R$, the letters $U(R)$ and $\rm{Nil}(R)$ are reserved for the unit group and the set of nilpotent elements in $R$, respectively. An element $r\in R$ is said to be {\it $n$-potent} for some $n\geq 2$, provided $r^n=r$. In the case where $n=2$, the element $r$ is known as {\it idempotent}. The set of all idempotents in $R$ is designed by ${\rm Id}(R)$. Standardly, $J(R)$ denotes the Jacobson radical of $R$.

Likewise, $\mathbb{Z}_n\cong \mathbb{Z}/(n)$ is identified with the ring of remainders of integer division by $n$, and we denote by ${\rm M}_n(R)$ the full matrix ring of size $n$ over $R$.

For all other unexplained explicitly notions and notations, we refer to \cite{lam1991first}.

The key instruments for our successful work are these.

\begin{definition}
Suppose $R$ is an arbitrary ring. Then, the following notions appeared in the existing literature:
\begin{enumerate}
 \item A ring $R$ is called {\it UU} if, for any $u \in U(R)$, $u \in 1 + {\rm Nil}(R)$ (see \cite{danchev2016rings}).
 \item A ring $R$ is called {\it $\pi$-UU} if, for any $u \in U(R)$, there exists $i \in N$ depending on $u$ such that $u^i \in 1 + {\rm Nil}(R)$ (see \cite{danchev2017exchange}).
 \item Let $n \in N$ be fixed. A ring $R$ is called $n$-UU if, for any $u \in U(R)$, $u^n \in 1 + {\rm Nil}(R)$ (see \cite{danchev2017exchange}).
 \item A ring R is called {\it strongly $n$-nil-clean} if every element of $R$ is a sum of an $n$-potent and a nilpotent that commute each other (see \cite{abyzov2019strongly}).
\end{enumerate}
\end{definition}

According to the above definitions, we observe that every UU ring is an $n$-UU ring and that every $n$-UU ring is a $\pi$-UU ring. Also, it is easy to see that if $R$ is a finite $\pi$-UU ring, then a number $n \in \mathbb{N}$ can be found such that $R$ is an $n$-UU ring.

Besides, strongly $n$-nil-clean rings are always {\it periodic} in the sense that their elements satisfy the eaquation $x^i = x^j$ for some $i>j\geq 1$.

The diagram listed below clearly demonstrates the above relations:

\begin{center}

\tikzset{every picture/.style={line width=0.75pt}} 

\begin{tikzpicture}[x=0.75pt,y=0.75pt,yscale=-1,xscale=1]

\draw    (270,75) -- (303,75) ;
\draw [shift={(306,75)}, rotate = 180] [fill={rgb, 255:red, 0; green, 0; blue, 0 }  ][line width=0.08]  [draw opacity=0] (8.04,-3.86) -- (0,0) -- (8.04,3.86) -- (5.34,0) -- cycle    ;
\draw    (449,75) -- (482,75) ;
\draw [shift={(485,75)}, rotate = 180] [fill={rgb, 255:red, 0; green, 0; blue, 0 }  ][line width=0.08]  [draw opacity=0] (8.04,-3.86) -- (0,0) -- (8.04,3.86) -- (5.34,0) -- cycle    ;
\draw    (430,130) -- (483,130) ;
\draw [shift={(486,130)}, rotate = 180] [fill={rgb, 255:red, 0; green, 0; blue, 0 }  ][line width=0.08]  [draw opacity=0] (8.04,-3.86) -- (0,0) -- (8.04,3.86) -- (5.34,0) -- cycle    ;
\draw    (250,130) -- (325,130.96) ;
\draw [shift={(328,131)}, rotate = 180.73] [fill={rgb, 255:red, 0; green, 0; blue, 0 }  ][line width=0.08]  [draw opacity=0] (8.04,-3.86) -- (0,0) -- (8.04,3.86) -- (5.34,0) -- cycle    ;
\draw    (206,90) -- (206.87,110) ;
\draw [shift={(207,113)}, rotate = 267.51] [fill={rgb, 255:red, 0; green, 0; blue, 0 }  ][line width=0.08]  [draw opacity=0] (8.04,-3.86) -- (0,0) -- (8.04,3.86) -- (5.34,0) -- cycle    ;
\draw    (519,90) -- (519.87,110) ;
\draw [shift={(520,113)}, rotate = 267.51] [fill={rgb, 255:red, 0; green, 0; blue, 0 }  ][line width=0.08]  [draw opacity=0] (8.04,-3.86) -- (0,0) -- (8.04,3.86) -- (5.34,0) -- cycle    ;
\draw    (379,90) -- (379.87,110) ;
\draw [shift={(380,113)}, rotate = 267.51] [fill={rgb, 255:red, 0; green, 0; blue, 0 }  ][line width=0.08]  [draw opacity=0] (8.04,-3.86) -- (0,0) -- (8.04,3.86) -- (5.34,0) -- cycle    ;

\draw (207.5,73) node   [align=left] {Strongly nil-clean};
\draw (378,73) node   [align=left] {Strongly $n$-nil-clean};
\draw (521.5,72) node   [align=left] {Periodic};
\draw (176,122) node [anchor=north west][inner sep=0.75pt]   [align=left] {UU rings};
\draw (346,123) node [anchor=north west][inner sep=0.75pt]   [align=left] {$(n-1)$-UU};
\draw (495,121) node [anchor=north west][inner sep=0.75pt]   [align=left] {$\displaystyle \pi -$UU};

\end{tikzpicture}

\end{center}

Note that the reverse implications hold if $R$ is strongly $\pi$-regular. Also, to understand this situation better, we have prepared the following table in order to characterize some plain examples of such types of rings:

\begin{table}[h!]
    \centering
    \begin{tabular}{c|c|c|c|c|c|c}
                                       & 2-UU     & 3-UU     & UU       & $\pi$-UU     & 6-UU         & 8-UU         \\ \hline
      $\mathbb{Z}$        & $\checkmark$ & $\times$     & $\times$ & $\checkmark$ & $\checkmark$ & $\checkmark$ \\
      $\mathbb{Z}_5$ & $\times$ & $\times$ & $\times$ & $\checkmark$ & $\times$ & $\checkmark$     \\
      $\mathbb{Z}_7$ & $\times$ & $\times$ & $\times$ & $\checkmark$ & $\checkmark$  & $\times$ \\
      ${\rm M}_2(\mathbb{Z}_2)$ & $\times$     & $\checkmark$ & $\times$ & $\checkmark$ & $\checkmark$     & $\times$ \\
      ${\rm M}_2(\mathbb{Z}_3)$ & $\times$     & $\times$ & $\times$ & $\checkmark$ & $\times$     & $\checkmark$
    \end{tabular}
\end{table}

Actually, our other basic motivating tool is somewhat to refine the results established in \cite{Kos}, where those rings $R$ are considered for which there is a fixed natural $n>1$ such that $u^n-1\in J(R)$ for each $u\in U(R)$. Note the well-known fact that $1+J(R)\subseteq U(R)$ always. Besides, the partial case when $U(R)=1+J(R)$ was comprehensively studied in \cite{D} paralleling to \cite{danchev2016rings}, by calling these rings as {\it rings with Jacobson units} or just {\it JU rings}. It is worthwhile noticing that the terminology in \cite{Kos} is slightly different, namely they called the aforementioned rings as {\it $n$-JU rings}.

\medskip

Thus, our further work is organized as follows: In the next section, we state and prove our achievements which are selected in three subsections. In the first one, we study a few characterizing properties of $n$-UU rings some of which will be applied in the subsequent two subsections (see Theorems~\ref{theorem1} and \ref{theorem2}). In the second one, we explore when matrix rings and some their generalizations are $n$-UU (see Theorem~\ref{thm 1.11}). In the third one, we investigate when group rings are $n$-UU and find suitable conditions under which this achieve (see Theorem~\ref{thm 1 group ring}). We finish off our considerations with a difficult query (see Problem~\ref{important}).

\section{Main Results and Examples}

Our chief results will be distributed on the next three subsections in the following manner:

\subsection{Characterization properties}

We start with the following useful technicality.

\begin{proposition}\label{UU}
Suppose $R$ is a ring of prime characteristic $p$ such that\\ $m-1 | p-1$. Then, $R$ is an $(m-1)$-UU ring if, and only if, every invertible element of the ring R is $m$-strongly nil clean.
\end{proposition}

\begin{proof}
Let $a\in U(R)$. Since $R$ is an $(m-1)$-UU ring, we have $b=1-a^{m-1} \in {\rm Nil}(R)$; and since $m-1 | p-1$, we have $b'=1-a^{p-1} \in {\rm Nil}(R)$. Therefore, there exist positive integers $k,k'$ such that $$(ab)^k=a^kb^k=b^ka^k=0=b'^{k'}a^k=a^kb'^{k'}=(ab')^{k'}.$$ Also, we can choose a number $n$ large enough so that $p^n>k,k'$. Since the characteristic of $R$ is the prime number $p$, it follows that
\begin{align*}
    1 &=(a^{m-1}+b)^{p^n}=(a^{m-1})^{p^n}+b^{p^n}=(a^{p^n})^{m-1}+b^{p^n}, \\
    1 &=(a^{p-1}+b')^{p^n}=(a^{p-1})^{p^n}+b'^{p^n}=(a^{p^n})^{p-1}+b'^{p^n}.
\end{align*}

Furthermore, suppose $e=a^{p^n}$, $f=b^{p^n}$ and $f'=b'^{p^n}$. Then, $e^{m-1}+f=1=e^{p-1}+f'$ and also $ef=(ab)^{p^n}=0=(ab')^{p^n}=ef'$. Therefore,
\begin{align*}
    e &=e(e^{m-1}+f)=e^p+ef=e^m, \\
    e &=e(e^{p-1}+f')=e^p+ef'=e^p.
\end{align*}

Likewise, on the other side, we have $$(a-e)^{p^n}=a^{p^n}-e^{p^n}=e-e^{p^n}=e-e=0,$$ so $a-e \in {\rm Nil}(R)$. Therefore, $a=e+(a-e)$, as required.
\end{proof}

As a valuable consequence, we extract the following.

\begin{corollary}
Suppose $R$ is a ring of prime characteristic $p$. Then, $R$ is a $\pi$-UU ring if, and only if, every invertible element of $R$ is periodic.
\end{corollary}

\begin{proof}
For every $a \in U(R)$, there exists a positive integer $k$ such that $b = a^k-1 \in {\rm Nil}(R)$. Let us assume that $b^s = 0$. We may take an integer $n$ large enough such that $p^n \ge s$. Thus, we obtain that
$$a^{kp^n} = (1 + b)^{p^n} = 1 + b^{p^n} = 1,$$ as needed.
\end{proof}

It is well known that if $K$ is a field and $K^{\ast}$ is the multiplicative group of $K$, then any finite subgroup of $K^{\ast}$ is cyclic (see, e.g., \cite[Proposition 3.7]{grove2012algebra}).

Now, the following technical claims and constructions give some further light on that matter.

\begin{lemma}\label{field}
Suppose $F$ is a field. If $F$ is $n$-UU, then $F$ is finite and $(|F|-1) | n$.
\end{lemma}

\begin{proof}
Let $f(x)=1-x^n \in F[x]$. Since $F$ is a field, the polynomial $f(x)$ has at most $n$ roots in $F^{\ast}$. So, if we suppose $A$ to be the set of all roots of $f$ in $F^{\ast}$, we will have $F^{\ast}=A$. Therefore, $|F^{\ast}|=|A|<n$.

On the other hand, as $F^{\ast}$ is a cyclic group, there exists $a\in F^{\ast} $such that $F^{\ast} =\langle a\rangle$. Since $a^n=1$, we get $o(a)|n$, and hence $n=o(a)q=|F^{\ast}|q$. Consequently, $|F^{\ast}||n$. Finally, $(|F|-1) | n$, as stated.
\end{proof}

\begin{lemma}\label{division ring}
Let $D$ be a division ring and $n\ge 2$. If $D$ is $n$-UU, then $D$ is a finite field and $(|D|-1) | n$.
\end{lemma}

\begin{proof}
Certainly, ${\rm Nil}(D)={0}$. So, for any $a \in D$, we have $a^n=1$, whence $a=a^{n+1}$. Furthermore, appealing to the famous Jacobson's theorem (cf. \cite{lam1991first}), we detect that $D$ must be commutative, and thus a field, as expected.

The second part follows at once from Lemma~\ref{field}.
\end{proof}

Our two next examples somewhat clarify the situation a bit.

\begin{example} \label{n mod 3}
Suppose $R = {\rm M}_2(\mathbb{Z}_2)$ and $n \equiv 0$ (mod 3). Then, we claim that $R$ is an $n$-UU ring.

In fact, since $U^6(R)={\rm I}_2$, it must be that $A^{6k}={\rm I}_2$ for every $A \in U(R)$ and $k\geq 1$. Thus, $({\rm I}_2-A^{3k})^2=0$. As $n=3k$, we infer that $R$ is an $n$-UU ring, as claimed.
\end{example}

\begin{example}
Suppose $R = {\rm M}_2(\mathbb{Z}_3)$ and $n \equiv 0$ (mod 8). Then, we claim that $R$ is an $n$-UU ring.

In fact, since $U^{24}(R)={\rm I}_2$, it must be that $A^{24k}={\rm I}_2$ for every $A \in U(R)$ and $k\geq 1$. Thus, $({\rm I}_2-A^{8k})^3=0$. As $n=8k$, we conclude that $R$ is an $n$-UU ring, as claimed.
\end{example}

The next assertion is well-known and is formulated here only for the sake of completeness and the reader's convenience.

\begin{lemma}\cite[Theorem 2(1)]{ohori1985strongly}\label{lemma 1}
A ring $R$ is strongly $\pi$-regular if, and only if, for any element $a \in R$, there exist $e = e^2 \in R$, $u \in U(R)$ and $w \in {\rm Nil}(R)$ such that $a = eu + w$ and $e$, $u$, $w$ all commutate each other.
\end{lemma}

We are now prepared to establish one of the main results.

\begin{theorem}\label{theorem1}
Let $R$ be a ring and $n \ge  2$ be a natural number. The following statements are equivalent:

(1) $R$ is strongly $n$-nil-clean.

(2) For each $a \in R$, $a = ev + b$, where $e^2 = e \in R$, $v^n = v \in U(R)$, and $b \in {\rm Nil}(R)$ with $ab = ba$ and $ev=ve$.

(3) For each $a \in R$, $a = ev + b$, where $e^2 = e \in R$, $v^n = v \in U(R)$, and $b \in {\rm Nil}(R)$ with $ab = ba$ and $ve=eve$.

(4) $a - a^n$ is nilpotent for every $a \in R$.

(5) For each $a \in R$, $a^{n-1}$ is a strongly nil-clean element in $R$.

(6) $R$ is strongly $\pi$-regular and an $(n-1)$-UU ring.
\end{theorem}

\begin{proof}

(1) $\Leftrightarrow$  (4). It follows directly from \cite[Theorem 8]{abyzov2022rings}.\\

(1) $\Rightarrow$  (2). Assume that (1) holds. Then, we can write every element $a\in R$ in the form $a=f+b$, where $f^n=f$, $b\in {\rm Nil}(R)$ and $ab=ba$. So, assuming $e=f^{n-1}$, we deduce $v=(1-f^{n-1})+f$, as required.\\

(2)  $\Rightarrow$  (3). It is clear.\\

(3)  $\Rightarrow$  (4). Clearly, we have $(ev)^n=v^{n-1}ev=ev$. Thus, we detect $a^n=ev+b'$, where $b'\in {\rm Nil}(R)$. Since $bb'=b'b$, it must be that $a-a^n\in {\rm Nil}(R)$, as required.\\

(1) $\Rightarrow$ (5). Suppose $a \in R$ is arbitrary. Then, we have that $a = e + q$, where $e^n = e, q \in {\rm Nil}(R)$ and $aq = qa$. With a simple calculation at hand, we obtain $a^{n-1} = e^{n-1} + p$, where $p \in {\rm Nil}(R)$ and $pa = ap$. Clearly, $(e^{n-1})^2 = e^{n-1}$, as required.\\

(5) $\Rightarrow$ (4). Suppose $a \in R$ is arbitrary. Then, we know that $a^{n-1}$ is strongly nil-clean in $R$. Thus, we can write $a^{n-1} = e + q$, where $e^2 = e, q \in {\rm Nil}(R)$ and $eq = qe$. So, it follows that $$a - a^n = a(1-e) - aq,$$ and hence $ea = ae$ utilizing \cite[Proposition 2.4]{kocsan2016nil}, so that $qa = aq$. Therefore, $(a - a^n)^{n-1} \in {\rm Nil(R)}$. Finally, $a - a^n \in {\rm Nil}(R)$, as needed.\\

(5) $\Rightarrow$ (6). Suppose $a \in R$ is arbitrary. Then, we see that $a^{n-1}$ is strongly nil-clean in $R$. Thus, we write $a^{n-1} = e + q$, where $e^2 = e$, $q \in {\rm Nil}(R)$ and $eq = qe$. Assuming that $q^k = 0$ for some $k\in \mathbb{N}$, we then have $a^{(n-1)k} = eu$, where $u \in U(R)$. So, using \cite[Proposition 1]{nicholson1999strongly}, $R$ is a strongly $\pi$-regular ring. Now, assume that $u \in U(R)$ is arbitrary. Write $u^{n-1} = e + q$, whence $$e = u^{n-1} - q \in {\rm Id}(R)\cap U(R)$$ forcing $e = 1$. Consequently, $R$ is an $(n-1)$-UU ring, as expected.\\

(6) $\Rightarrow$ (5). Suppose $a \in R$ is arbitrary. Since $R$ is strongly $\pi$-regular, Lemma \ref{lemma 1} gives that $$a = eu + q,$$ where $e^2 = e$, $u \in U(R)$, $q \in {\rm Nil}(R)$ and $e$, $u$, $q$ commutate. Clearly, one calculates that $$a^{n-1} = eu^{n-1} + p,$$ where $p \in {\rm Nil}(R)$. By assumption, let us write $u^{n-1} = 1 + b$ for some $b \in {\rm Nil}(R)$. Therefore, $$a^{n-1} = e(1+b) + p = e + eb + p.$$ Since $eu=ue$, we have $eb=be$, so $eb \in {\rm Nil}(R)$. Also, as $e$, $u$, $q$ commute, we can get $eb+p \in {\rm Nil}(R)$. Finally, $a^{n-1}$ is a strongly nil-clean element, as pursued.
\end{proof}

It is worthwhile noticing that it was shown in \cite{cui2020some} that any strongly $\pi$-regular $\pi$-UU ring is periodic, and vice versa. Thereby, the equivalence (1) $\iff$ (6) could be treated as an extension of the mentioned fact. Besides, in the case when $n=2$, we directly derive from the preceding theorem the following well-known fact.

\begin{corollary}\cite[Corollary 3.11]{diesl2013nil}
Let $R$ be a ring. Then, $R$ is strongly nil-clean if, and only if, $R$ is a strongly $\pi$-regular UU ring.
\end{corollary}

We will denote the least common multiple of integers $m_1, \ldots , m_k$ by the notation $[m_1, \ldots , m_k]$.

\medskip

As a consequence to the last theorem, we yield the following statement in which point (i) is new and outside of the original.

\begin{corollary}\cite[Theorem 2]{abyzov2022rings} \label{matrix n-UU}
Let $F$ be a finite field, $m, n \in \mathbb{N}$, and $n > 1$. The following statements are equivalent:
\begin{enumerate}
    \item ${\rm M}_m(F)$ is an $(n-1)$-UU ring.
    \item ${\rm M}_m(F)$ is an strongly n-nil-clean.
    \item  $N = [|F| -1, |F|^2 -1, ..., |F|^m -1]$ is a divisor of $n-1$.
\end{enumerate}
\end{corollary}

We continue our work with some more technical claims.

\begin{proposition}\label{2 in J(R)}
Let $R$ be an $n$-UU ring, where $n$ is an odd number. Then, the element $2$ is central nilpotent and, as such, is always contained in $J(R)$.
\end{proposition}

\begin{proof}
We know that the element $-1$ is a unit and also that $(-1)^n = -1$. Therefore, based on our assumption, we have $-1 = 1 + q$, where $q \in {\rm Nil}(R)$. Hence, $2 = -q \in {\rm Nil}(R)$, as promised.
\end{proof}

If $n$ is even, it is not necessarily true that $2 \in {\rm Nil}(R)$, because $\mathbb{Z}$ is always a $2$-UU ring, but $2$ is definitely {\it not} an element of ${\rm Nil}(\mathbb{Z})$.

\medskip

We now partially answer in the positive a question posed in \cite{cui2020some}.

\begin{proposition}\label{prop 1.3}
Let $R$ be an $n$-UU ring, where $n$ is an odd number. Then, $J(R)$ is nil.
\end{proposition}

\begin{proof}
Choose $a\in J(R)$. Therefore, $1-a\in U(R)$, and hence $(1-a)^n=1+q$, where $q\in {\rm Nil}(R)$. Thus, we have $$na - \cdots +a^n=-q,$$ and since $2 \in {\rm Nil}(R)$ and $n$ is odd, a simple check shows that it must be that $n \in U(R)$. So, we find that $$a - \cdots + n^{-1}a^{n}=-n^{-1}q  \in {\rm Nil}(R),$$ whence $$a(1+af(a))=q' \in {\rm Nil}(R),$$ where $q'=n^{-1}q$, as asserted.

Furthermore, note that $1+af(a) \in U(R)$ and $$q(1+af(a))=(1+af(a))q,$$ because $q$ is a combination of sums of powers of the element $a$. Consequently, we deduce that $a \in {\rm Nil}(R)$, as wanted.
\end{proof}

Before proceed with the case of $n$-UU rings for even number $n$, we establish the following statement.

\begin{proposition}
Let R be a $2^k$-UU ring. Then, $J(R)$ is nil.
\end{proposition}

\begin{proof}
Using induction on $k$, it can easily be shown that
$$(1\pm a)^{2^k}=1\pm 2^k \sum_{i=0}\lambda_{2i+1}a^{2i+1}+ \sum_{i=1}\lambda'_{2i}a^{2i}+a^{2^k},$$
where $\lambda_1=1$. Now, if $a \in J(R)$, we receive that
$$2^k \sum_{i=0}\lambda_{2i+1}a^{2i+1} \in {\rm Nil}(R),$$
and because $\lambda_1=1$ and $a \in J(R)$, one follows that $2^ka \in {\rm Nil}(R)$, so that $2a \in {\rm Nil}(R)$.

On the other side, by induction on $k$, it can plainly be shown that
$$(1+a)^{2^k}=1+2ag(a)+a^{2^k},$$ where $g(t) \in \mathbb{N}[t]$.
Therefore, $$(1+a)^{2^k}=1+2ag(a)+a^{2^k}=1+q,$$ where $q \in {\rm Nil}(R)$, hence $a^{2^k} \in {\rm Nil}(R)$, and hence $a \in {\rm Nil}(R)$, as asked for.
\end{proof}

As an automatic consequence, we derive:

\begin{corollary}\cite[Proposition 2.1.]{danchev2017exchange}\label{prop 1.4}
Let $R$ be a $2$-UU ring. Then, $J(R)$ is nil.
\end{corollary}

We are now in a position to generalize this to the following affirmations.

\begin{proposition}\label{prop 1.6}
Let $R$ be a $6$-UU ring. Then, $J(R)$ is nil.
\end{proposition}

\begin{proof}
Given $a \in J(R)$, we have $1\pm a, 1\pm a^2 \in U(R)$. Thus, one verifies that:

\begin{align}
    (1-a)^6 &=1-6a+15a^2-20a^3+15a^4-6a^5+a^6 =1+q_0 \label{eq 1}\\
    (1+a)^6 &=1+6a+15a^2+20a^3+15a^4+6a^5+a^6 =1+q_1 \label{eq 2}\\
    (1-a^2)^6 &=1-6a^2+15a^4-20a^6+15a^8-6a^{10}+a^{12} =1+q_2 \label{eq 3}\\
    (1+a^2)^6 &=1+6a^2+15a^4+20a^6+15a^8+6a^{10}+a^{12} =1+q_3 \label{eq 4},
\end{align}
where all $q_i \in {\rm Nil}(R)$ and it is evident that $q_iq_j = q_jq_i$. Thus, with equations (\ref{eq 1}) and (\ref{eq 2}) at hand, we receive:

\medskip

Eq (01): $30a^2+30a^4+2a^6=q' \in {\rm Nil}(R)$

\medskip

\noindent and, according to equations (\ref{eq 3}) and (\ref{eq 4}), we receive:

\medskip

Eq (02): $30a^4+30a^8+2a^{12}=q'' \in {\rm Nil}(R)$.

\medskip

\noindent Since $aq'=q'a$, if we multiply by $a^6$ in equation Eq (01), we will obtain:

\medskip

Eq (03): $30a^8+30a^{10}+2a^{12}=p \in {\rm Nil}(R)$.

\medskip

\noindent Since $q''p=pq''$, it must be that $q''-p \in {\rm Nil}(R)$, so with the aid of Eq (03) and Eq (02) we arrive at the relation $30a \in {\rm Nil}(R)$ which insures $2a \in {\rm Nil}(R)$. Therefore, equalities (\ref{eq 1}) and (\ref{eq 3}) given above help us to get that

\begin{align}
15a^2+15a^4+a^6&=p_0 \in {\rm Nil}(R), \label{eq 5}\\
15a^4+15a^8+a^{12}&=p_1 \in {\rm Nil}(R). \label{eq 6}
\end{align}

Furthermore, if we multiply by $a^6$ in equation (\ref{eq 5}), we will obtain:

\medskip

Eq (04): $15a^8+15a^{10}+a^{12}=p_2 \in {\rm Nil}(R)$.

\medskip

Now, in view of equation (\ref{eq 6}) and Eq (04), we can derive that $15a \in {\rm Nil}(R)$. Therefore, bearing in mind that both elements $2a$ and $15a$ are nilpotents, the equation (\ref{eq 1}) quoted above ensures that $a^6 \in {\rm Nil}(R)$. Finally, $a \in {\rm Nil}(R)$, and thus we are done.
\end{proof}

\begin{proposition}\label{prop 1.7}
Let R be a $10$-UU ring. Then, $J(R)$ is nil.
\end{proposition}

\begin{proof}
Given $a \in J(R)$, we have $(1\pm a), (1\pm a^3) \in U(R)$, and similarly to Proposition \ref{prop 1.6} we obtain:

\begin{align}
    20a+240a^3+504a^5+240a^7+20a^9 &=q_1 \label{eq 7}\\
    20a^3+240a^9+504a^{15}+240a^{21}+20a^{27} &=q_2 \label{eq 8}\\
    90a^2+420a^4+420a^6+90a^8+2a^{10} &=q_3 \label{eq 9},
\end{align}
where $q_i \in {\rm Nil}(R)$, and $q_iq_j=q_jq_i$. Since $aq_i=q_ia$, if we multiply by $a^{10}$ in equation (\ref{eq 7}), we will obtain:

\medskip

Eq (05): $20a^{11}+240a^{13}+504a^{15}+240a^{17}+20a^{19}=q_4 \in {\rm Nil}(R)$.

\medskip

\noindent Since $q_iq_j=q_jq_i$, it must be that $q_4 -q_2 \in {\rm Nil}(R)$, so with the aid of equation (\ref{eq 8}) and Eq (05) we arrive at the relation $20a \in {\rm Nil}(R)$, so $10a \in {\rm Nil}(R)$. Now, from equation (\ref{eq 9}), we have $2a \in {\rm Nil}(R)$. But $1-(1-a)^{10}$, $1-(1-a^3)^{10} \in {\rm Nil}(R)$, and so we conclude that:

\begin{align}
    45a^2+45a^8 +a^{10} &=p_1 \in {\rm Nil}(R) \label{eq 10}\\
    45a^{6}+45a^{24} +a^{30} &=p_2 \in {\rm Nil}(R) \label{eq 11},
\end{align}

\noindent Since $ap_i=p_ia$, if we multiply by $a^{20}$ in equation (\ref{eq 10}), we will have:

\medskip

Eq (06): $45a^{22}+45a^{28} +a^{30}=p_3 \in {\rm Nil}(R)$.

\medskip

\noindent Since $p_ip_j=p_jp_i$, it must be that $p_3 -p_2 \in {\rm Nil}(R)$, so with the aid of equation (\ref{eq 11}) and Eq (06) we arrive at the relation $45a \in {\rm Nil}(R)$, and thus equation (\ref{eq 10}) implies that $a^{10} \in {\rm Nil}(R)$. Therefore, $a \in {\rm Nil}(R)$, and so we are set.
\end{proof}

We continue our work with the following statements.

\begin{proposition}\label{pro 1}
Let $R$ be an $n$-UU ring and $k \in \mathbb{N}$ such that $n|k$. Then, $R$ is a $k$-UU ring.
\end{proposition}

\begin{proof}
Since $R$ is an $n$-UU ring, for any $u \in U(R)$ we may write that $u^n = 1 + q$, where $q \in {\rm Nil}(R)$. Since $n|k$, there exists an integer $t$ such that $k = tn$. Thus, $$u^k = (u^n)^t = (1 + q)^t = 1 + q',$$ where $q' = (1+q)^t-1$ which is obviously nilpotent. Therefore, $u^k = 1 + q'$, where $q' \in {\rm Nil}(R)$. Hence, $R$ is a $k$-UU ring, as stated.
\end{proof}

\begin{proposition} \label{n odd two item}
Let $n$ be an odd integer. Then, $R$ is an $n$-UU ring if, and only if, the next two points are valid:
    \begin{enumerate}
        \item 2 is nilpotent.
        \item $R$ is a $2^kn$-UU ring for every positive integer $k$.
    \end{enumerate}
\end{proposition}

\begin{proof}
"$\Rightarrow$". According to Proposition \ref{2 in J(R)} and  \ref{pro 1}, there is nothing to prove.\\

"$\Leftarrow$". Given conditions (i) and (ii) hold. Supposing $a \in U(R)$, then $a^{2^kn}-1 =q \in {\rm Nil}(R)$, so a routine manipulation leads to
$$(1-a^n)^{2^k}=1-2af(a)+a^{2^kn}=q +2(1-af(a)) \in {\rm Nil}(R),$$ where $f(a)\in \mathbb{Z}[t]$, as wanted.
\end{proof}

\begin{corollary}\label{cor UU}
A ring $R$ is a UU ring if, and only if,
    \begin{enumerate}
        \item 2 is nilpotent.
        \item R is a $2^k$-UU ring for every positive integer $k$.
    \end{enumerate}
\end{corollary}

\begin{corollary}\cite[Proposition 2.2]{abdolyousefi2018rings}
A ring $R$ is a UU ring if, and only if,
    \begin{enumerate}
        \item 2 is nilpotent.
        \item $R$ is a 2-UU ring.
    \end{enumerate}
\end{corollary}

In \cite{cui2020some}, it was asked of whether or {\it not} an $\pi$-UU ring that is nil clean is a strongly $\pi$-regular. We answer a specific case of this question in the following lemma (compare with \cite{danchev2016rings} too).

\begin{lemma}
A ring $R$ is a strongly nil-clean ring if, and only if,
 \begin{enumerate}
     \item $R$ is a nil-clean ring;
     \item R is a $2^k$-UU ring for every positive integer $k$.
 \end{enumerate}
In particular, $R$ is a strongly $\pi$-regular ring.
\end{lemma}

\begin{proof}
"$\Rightarrow$". It is straightforward.\\

"$\Leftarrow$". Since $R$ is a  nil-clean ring, we know that $2 \in {\rm Nil}(R)$ (see \cite{diesl2013nil}). Thus, by the usage of Corollary~\ref{cor UU}, $R$ will be a UU ring. With the help of \cite[Corollary 2.13]{vster2016rings}, we have that $R$ is an NR-ring (that is, ${\rm Nil}(R)$ forms an ideal of $R$).

On the other hand, since $R$ is exchange, by \cite[Corollary 2.17]{chen2015linearly} we can deduce that the factor-ring $R/J(R)$ is reduced. Hence, ${\rm Nil}(R) \subseteq J(R)$. Moreover, \cite[Proposition 3.16]{diesl2013nil} applies to get that $J(R) \subseteq {\rm Nil}(R)$. Therefore, ${\rm Nil}(R)=J(R)$ is an ideal. Finally, by using \cite[Theorem 2.3]{chen2017strongly}, $R$ is a strongly nil-clean ring, as desired.
\end{proof}

We now begin to explore some crucial characteristic properties of $n$-UU rings starting with the following two technicalities.

\begin{lemma}\label{sub}
(a) An arbitrary subring $S$ of an $n$-UU ring $R$ is again an $n$-UU ring.

(b) A finite direct product $R=R_1 \times \cdots \times R_m $ is $n$-UU if, and only if, each direct component $R_i$ is $n$-UU.
\end{lemma}

\begin{proof}
(a) Let $u \in U(S)$. Then, $u+(1_R - 1_S) \in U(R)$ as $$[u+(1_R - 1_S)][u^{-1}+(1_R - 1_S)]=1_R.$$
On the other hand, since $R$ is an $n$-UU ring, we write $$u^n+(1_R - 1_S)=[u+(1_R - 1_S)]^n= 1_R +q,$$ where $q \in {\rm Nil}(R)$. Therefore, $u^n= 1_S +q$ and note that $q \in S$, because $q=u^n-1_S \in S$.\\

(b) The necessity follows directly from (a).

For sufficiency, if all $R_i$ are $n$-UU rings and $(u_1, u_2,\ldots,u_m) \in U(R)$, then it follows coordinate-wise that each $u_i \in U(R_i)$. Therefore, $\mathrm{u}_{i}^{n} =1+q_i$, where $q_i \in {\rm Nil}(R_i)$. Assume that $\mathrm{q}_{i}^{n_i}=0$ and $k:={\rm max}\{n_i\}\mathrm{}_{i=1}^{m}$. Consequently,
$$(u_1, \ldots,u_m)^n=(1,\ldots,1)+(q_1,q_2,\ldots,q_m)$$ such that $(q_1,q_2,\ldots,q_m)^k=0$ for some $k\in \mathbb{N}$, as required.
\end{proof}

As an intriguing consequence, we obtain:

\begin{corollary} \label{corner}
If $R$ is an $n$-UU ring for some $n \in \mathbb{N}$, then the corner ring $eRe$ is also $n$-UU for any $e=e^2 \in R$. In addition, if $e$ is a central idempotent, $R$ is an $n$-UU ring, provided both $eRe$ and $(1-e)R(1-e)$ are so.
\end{corollary}

\begin{proof}
Since $eRe$ is a subring of $R$ (certainly, {\it not} necessarily unital), Lemma~\ref{sub}(a) is applicable to get the claim.

If now $e$ is a non-trivial central idempotent, we have the decomposition $R=eRe \oplus (1-e)R(1-e)$, so that Lemma~ \ref{sub}(b) applies to get the desired claim that $eRe$ is an $n$-UU ring.
\end{proof}

To treat the general situation, one may also ask the question: Does it follow that $R$ is an $n$-UU ring, provided both $eRe$ and $(1-e)R(1-e)$ are $n$-UU rings?

\medskip

We note that the infinite products of $n$-UU rings may {\it not} be an $n$-UU ring. For instance, letting $R =\prod_{k=1}^{\infty} \mathbb{Z}_{2^k}$, then $R$ is not an $n$-UU ring, because $a = (1, 3,..., 3,...) \in U(R)$, but the degree $a^n$ is {\it not} an unipotent. However, each $\mathbb{Z}_{2^k}$ is an $n$-UU ring for all $k \in \mathbb{N}$ (note that the ring $\mathbb{Z}_{2^k}$ is even strongly nil-clean).

\medskip

Consulting with Theorem \ref{theorem1}, a ring $R$ is strongly $n$-nil-clean if, and only if, it is both strongly $\pi$-regular and $(n-1)$-UU. In \cite{kocsan2016nil}, the authors provides several results concerning strongly $n$-nil-clean rings for various numbers $n$. In the following result, we examine semi-local $(n-1)$-UU rings.

\begin{theorem}\label{theorem2}
For a ring $R$, consider the following three conditions:
     \begin{enumerate}
         \item $R$ is an $(n-1)$-UU ring.
         \item $R$ is strongly $n$-nil-clean.
         \item $J(R)$ is nil and $R/J(R)$ is a subdirect product of the rings ${\rm M}_m(F)$, where $F$ is a finite field and $(|F|^i -1) | (n -1)$ for every $1 \le i \le m$.
     \end{enumerate}
Then, $(iii)$ $\Rightarrow$ $(ii)$$\Rightarrow$ $(i)$. The converse is also true, provided $R$ is semi-local and $n$ is either an even number or is of the form $n = 2^k + 1$, $n = 7$, or $n = 11$.
\end{theorem}

\begin{proof}
(ii) $\Leftrightarrow$ (iii). The equivalence is exactly \cite[Theorem 8]{abyzov2022rings}.

The implication (ii) $\Rightarrow$ (i) is clear. \\

(i) $\Rightarrow$ (iii).  Exploiting Proposition \ref{prop 1.3}, $J(R)$ is nil. As $R$ is semi-local, we may write that $$R/J(R) = {\rm M}_{n_1}(D_1) \oplus {\rm M}_{n_2}(D_2) \oplus \cdots \oplus {\rm M}_{n_k}(D_k),$$ where each $D_i$ is a division ring. Since the quotient-ring $R/J(R)$ is an $(n-1)$-UU ring, each of the direct components ${\rm M}_{n_i}(D_i)$ is also $(n-1)$-UU. And since we have the inclusions $D_i \subseteq {\rm M}_{n_i}(D_i)$, all of the $D_i$'s must also be $(n-1)$-UU rings owing to Lemma~\ref{sub} (a). Hence, Proposition \ref{division ring} allows us to conclude that all of the $D_i$'s must be finite fields. Thus, Corollary \ref{matrix n-UU} is applicable to get the result.
\end{proof}

We know that matrices over any ring are neither UU nor $2$-UU (see, e.g., \cite{danchev2016rings} and \cite{cui2020some}). So, a question which automatically arises is what can we say about matrices over rings for an arbitrary natural number $n>2$? In what follows, we try to answer this question in some different cases.

\begin{example} \label{ex matrix 3-UU}
For any ring $R\neq 0$ and any integer $n\ge 3$, ${\rm M}_n(R)$ is {\it not} a $3$-UU ring.
\end{example}

\begin{proof}
Based on Corollary \ref{corner}, it is sufficient to show that ${\rm M}_3(R)$ is not a $3$-UU ring. Assume, in a way of contradiction, that ${\rm M}_3(R)$ is a $3$-UU ring. Since
    $$\begin{pmatrix}
        1 & 1 & 1\\
        1 & 1 & 0\\
        1 & 0 & 0
    \end{pmatrix}^{-1}=
    \begin{pmatrix}
        0 & 0 & 1\\
        0 & 1 & -1\\
        1 & -1 & 0
    \end{pmatrix} \in {\rm GL}_3(R),
    $$
we have
$$
\begin{pmatrix}
        1 & 1 & 1\\
        1 & 1 & 0\\
        1 & 0 & 0
    \end{pmatrix}^{3}-
\begin{pmatrix}
        1 & 0 & 0\\
        0 & 1 & 0\\
        0 & 0 & 1
    \end{pmatrix}=
\begin{pmatrix}
        5 & 5 & 3\\
        5 & 3 & 2\\
        3 & 2 & 0
    \end{pmatrix} \in {\rm Nil}({\rm M}_3(R)).
$$
But, by Proposition \ref{2 in J(R)}, $2 \in {\rm Nil}(R)$, whence one calculates that
$$\begin{pmatrix}
        5 & 5 & 3\\
        5 & 3 & 2\\
        3 & 2 & 0
    \end{pmatrix}^{-1}=\frac{1}{13}\begin{pmatrix}
        -4 & 6 & 1\\
        6 & -9 & 5\\
        1 & 5 & -10
    \end{pmatrix},$$
which is a contradiction. This completes the proof.
\end{proof}

In the above example, we cannot assume $n \ge 2$; indeed, this is because, according to Example \ref{n mod 3}, the ring ${\rm M}_2(\mathbb{Z}_2)$ is $3$-UU.

\begin{example} \label{ex matrix 4-UU}
For any ring $R\neq 0$ and any integer $n\ge 2$, ${\rm M}_n(R)$ is {\it not} a $4$-UU ring.
\end{example}

\begin{proof}
It suffices to show that ${\rm M}_2(R)$ is not a $4$-UU ring. Suppose
$A=
    \begin{pmatrix}
        1 & 1 \\
        1 & 0
    \end{pmatrix},$
and $S = 1_R\cdot \mathbb{Z}$. Apparently, $S$ is a commutative subring of $R$ and $A \in {\rm GL}_2(R)$. Then, one computes that ${\rm det}(A^4-{\rm I})=5 \in {\rm Nil}(S)$, so $4\in U(R)$. Now, supposing
$B=\begin{pmatrix}
        2 & 2 \\
        2 & -2
    \end{pmatrix}$,
we have $B^{-1}=\frac{1}{4} \begin{pmatrix}
        1 & 1 \\
        1 & -1
    \end{pmatrix}.$
Therefore, plain calculations lead us to ${\rm det}(B^4-{\rm I})=3^2.7 \in {\rm Nil}(S)$, but $(5, 3^2.7)=1$ and hence $1_R \in {\rm Nil}(R)$, a contradiction. This finishes the proof.
\end{proof}

\begin{example} \label{ex matrix 5-UU}
For any ring $R\neq 0$ and any integer $n\ge 2$, ${\rm M}_n(R)$ is {\it not} a $5$-UU ring.
\end{example}

\begin{proof}
Assume the contrary that ${\rm M}_2(R)$ is a $5$-UU ring. Since
$$\begin{pmatrix}
        1 & 1 \\
        1 & 0
    \end{pmatrix} \in {\rm GL}_2(R),
    $$
we extract
$$
\begin{pmatrix}
        1 & 1 \\
        1 & 0
    \end{pmatrix}^{5}-
\begin{pmatrix}
        1 & 0 \\
        0 & 1
    \end{pmatrix}=
\begin{pmatrix}
        7 & 5\\
        3 & 2
    \end{pmatrix} \in {\rm Nil}({\rm M}_2(R)).
$$
But, by Proposition \ref{2 in J(R)}, $2 \in {\rm Nil}(R)$, so one verifies that
$$\begin{pmatrix}
        7 & 5\\
        3 & 2
    \end{pmatrix}^{-1}=\frac{1}{11}\begin{pmatrix}
        -2 & 5\\
         5 & -7
    \end{pmatrix},$$
a contradiction. This concludes the proof.
\end{proof}

\begin{example} \label{ex matrix 6-UU}
For any ring $R\neq 0$ and any integer $n\ge 3$, ${\rm M}_n(R)$ is {\it not} a $6$-UU ring.
\end{example}

\begin{proof}
Suppose ${\rm M}_3(R)$ is a $6$-UU ring. Then, thanks to Corollary \ref{corner}, we have that ${\rm M}_2(R)$ is a $6$-UU ring. With the settings
    $A=\begin{pmatrix}
        1 & 1\\
        1 & 0
    \end{pmatrix}$
and $S=1_R\cdot \mathbb{Z}$ at hand, one checks that $S$ is a commutative subring of $R$ and $A \in {\rm GL}_2(R)$. This gives that ${\rm det}(A^6-{\rm I})=2^4 \in {\rm Nil}(S)$. Therefore, $2 \in {\rm Nil}({\rm M}_3(R))$, and thus Proposition \ref{n odd two item} implies that ${\rm M}_3(R)$ is a $3$-UU ring. This, however, contradicts Example \ref{ex matrix 3-UU} and so our claim follows, as promised.
\end{proof}

Notice the interesting fact that, for $n=2,3,4,5,6$, it is easy to characterize semi-local $n$-UU rings by using Theorem \ref{theorem2} and the above examples.

\begin{lemma}\label{nil ideal}
Let $I$ be a nil-ideal of a ring $R$. Then, $R$ is $n$-UU if, and only if, so is $R/I$.
\end{lemma}

\begin{proof}
Assume that $R$ is an $n$-UU ring. Let $\bar{u}:=u+I \in U(R/I)$. So, $u \in U(R)$, and hence we have $u^n=1+q$, where $q \in {\rm Nil}(R)$. Therefore, $\bar{u}^n=\bar{1}+\bar{q}$, as required.

Conversely, assume that $R/I$ is an $n$-UU ring. Let $u \in U(R)$. Then, $\bar{u} \in U(R/I)$. Thus, $$\overline{u^n-1}=\bar{u}^n-\bar{1} \in {\rm Nil}(R/I).$$ As $I$ is nil, we can get $u^n -1 \in {\rm Nil}(R)$ meaning that $R$ is $n$-UU, as needed.
\end{proof}

As an immediate consequence of the results established above, we extract:

\begin{corollary}\label{cor 1.10}
\begin{enumerate}
        \item Suppose $n$ is a positive odd integer. Then, $R$ is an $n$-UU ring if, and only if, $J(R)$ is nil and $R/J$ is $n$-UU.
        \item Let $n=2^k,6,10$. Then, $R$ is an $n$-UU ring if, and only if, $J(R)$ is nil and $R/J$ is $n$-UU.
\end{enumerate}
\end{corollary}

\subsection{Matrix rings and generalizations}

Let $A, B$ be two rings and let $M, N$ be the $(A, B)$-bimodule and $(B, A)$-bimodule, respectively. Also, we consider the bilinear maps $\phi : M\otimes_B N \to A$ and $\psi : N\otimes_AM \to B$ that apply to the following properties
$$Id_M \otimes_B \psi = \phi \otimes_A Id_M, \quad Id_N \otimes_A \phi = \psi \otimes_B Id_N .$$
For $m \in M$ and $n \in N$, we define $mn := \phi(m \otimes n)$ and $nm := \psi(n \otimes m)$.
Thus, the 4-tuple
$R= \begin{pmatrix}
A & M \\
N & B
\end{pmatrix}$
becomes to an associative ring equipped with the obvious matrix operations, which is called a {\it Morita context ring}. Denote the two-sided ideals $Im\phi$ and $Im\psi$ to $MN$ and $NM$, respectively, that are called the {\it trace ideals} of the Morita context (compare with \cite{ABD} as well).

\medskip

We are managed to prove the following assertion.

\begin{theorem}\label{thm 1.11}
Suppose $n$ is a positive integer such that either $n$ is odd or $n=2^k,6,10$. Let $R= \begin{pmatrix}
A & M \\
N & B
\end{pmatrix}$
be a Morita context ring such that $MN$ and $NM$ are nilpotent ideals of $A$ and $B$, respectively. Then, $R$ is $n$-UU if, and only if, both $A$ and $B$ are $n$-UU rings.
\end{theorem}

\begin{proof}
As both $A$ and $B$ are subrings of $R$, in accordance with Lemma~\ref{sub} (a) they are $n$-UU ring.\\

If now $A,B$ are two $n$-UU rings, then by virtue of Corollary \ref{cor 1.10} we have that $J(A)$ and $J(B)$ are both nil-ideals and also that $A/J(A)$ and $B/J(B)$ are both $n$-UU rings. On the other side, referring to \cite[Theorem 13(1)]{danchev2017exchange}, we compute
$$J(R)=
\begin{pmatrix}
J(A) & M \\
N & J(B)
\end{pmatrix},$$ and
$$R/J(R) \cong A/J(A)\times B/J(B).$$
Since Lemma~\ref{sub} (b) assures that the direct product of $n$-UU rings is again $n$-UU, the ring $R/J(R)$ is also $n$-UU. Furthermore, according to Corollary \ref{cor 1.10}, it is sufficient to show that $J(R)$ is a nil-ideal. To that purpose, suppose
$$S= \begin{pmatrix}
a_1 & m \\
n & a_2
\end{pmatrix} \in J(R)$$
such that $a_1^n=a_2^n=0$ for some $n \in \mathbb{N}$. Then, one finds that
$$S^n \in
\begin{pmatrix}
 MN& M \\
N & NM
\end{pmatrix},
$$
and hence, we easily obtain by induction on $k$ that
$$\begin{pmatrix}
 MN& M \\
N & NM
\end{pmatrix}^{2k}=
\begin{pmatrix}
 (MN)^k& (MN)^kM \\
(NM)^kN & (NM)^k
\end{pmatrix}.
$$
Therefore, $J(R)$ is nil, as asserted.
\end{proof}

Given a ring $R$ and a central element $s$ of $R$, the $4$-tuple
$\begin{pmatrix}
 R& R \\
 R& R
\end{pmatrix}$
becomes a ring with addition defined componentwise and with multiplication defined by
$$
\begin{pmatrix}
 a_1& x_1 \\
 y_1& b_1
\end{pmatrix}
\begin{pmatrix}
  a_2& x_2 \\
 y_2& b_2
\end{pmatrix}=
\begin{pmatrix}
 a_1a_2 + sx_1y_2& a_1x_2 + x_1b_2 \\
 y_1a_2 + b_1y_2& sy_1x_2 + b_1b_2
\end{pmatrix}.
$$
This ring is denoted by $K_s(R)$. A Morita context
$
\begin{pmatrix}
  A& M \\
 N& B
\end{pmatrix}
$
with $A = B = M = N = R$ is called a generalized matrix ring over $R$. It was observed in \cite{krylov2008isomorphism} that a ring $S$ is a generalized matrix ring over $R$ if, and only if, $S = K_s(R)$ for some $s \in C(R)$, the center of $R$. Here $MN = NM = sR$, so that $MN \subseteq J(A) \Longleftrightarrow  s \in J(R)$, $NM \subseteq J(B) \Longleftrightarrow  s \in  J(R)$, and $MN, NM$ are nilpotents $\Longleftrightarrow  s$ is a nilpotent.

\medskip

Thus, Theorem \ref{thm 1.11} yields the following two direct consequences.

\begin{corollary}
Suppose $n$ is a positive integer such that either $n$ is odd or $n=2^k,6,10$, R is a ring, and $s \in C(R)\cap {\rm Nil}(R$). Then the formal matrix ring $K_s(R)$ is an $n$-UU ring if, and only if, $R$ is $n$-UU ring.
\end{corollary}

\begin{corollary}
Suppose $n$ is a positive integer such that $n$ is odd, then the formal matrix ring $K_2(R)$ is an $n$-UU ring if, and only if, $R$ is $n$-UU ring.
\end{corollary}

Let $R, S$ be two rings, and let $M$ be an $(R, S)$-bimodule such that the operation $(rm)s = r(ms$) is valid for all $r \in R$, $m \in M$ and $s \in S$. Given such a bimodule $M$, we can put

$$
T(R, S, M) =
\begin{pmatrix}
 R& M \\
 0& S
\end{pmatrix}
=
\left\{
\begin{pmatrix}
 r& m \\
 0& s
\end{pmatrix}
: r \in R, m \in M, s \in S
\right\},
$$
where it forms a ring with the usual matrix operations. The so-stated formal matrix $T(R, S, M)$ is called a {\it formal triangular matrix ring}.

\medskip

The following result can quickly be verified by using Theorem \ref{thm 1.11}.

\begin{corollary}
Suppose $n$ is a positive integer such that either $n$ is odd or $n=2^k,6,10$. Let $R$, $S$ be two rings and let $M$ be an $(R, S)$-bimodule, and $N$ a bimodule over $R$.
\begin{enumerate}
    \item $T(R, S, M)$ is an $n$-UU ring if, and only if, $R$ and $S$ are $n$-UU rings.\\
    \item The trivial extension $R \propto N$ is an $n$-UU ring if, and only if, $R$ is an $n$-UU ring.\\
    \item For $n \ge 2$, ${\rm T}_n(R)$ is an $n$-UU ring if, and only if, $R$ is $n$-UU.\\
    \item For $n \ge 2$, $R[x]/(x^n)$ is an $n$-UU ring if, and only if, $R$ is an $n$-UU ring.
\end{enumerate}
\end{corollary}

\begin{proposition}
The power series ring $R[[t]]$ is not $n$-UU for any ring $R$.
\end{proposition}

\begin{proof}
Referring to \cite[Exersice 6.5]{lam1991first}, we know that $$J(R[[x]])=\{a+xf(x) | x \in J(R) , f(x) \in R[[x]] \}.$$ So, $1+x \in U(R)$. But $(1+x)^n$ is not an unipotent for every $n \in \mathbb{N}$, as required.
\end{proof}

According to \cite{danchev2016rings} and \cite{cui2020some}, a commutative ring $R$ is a UU (resp. $\pi$-UU) ring if, and only if, the ring of polynomials $R[t]$ is also a UU (resp. $\pi$-UU) ring. Now, we aim to extend these results to $n$-UU rings. The prime radical $N(R)$ of a ring R is defined as the intersection of all prime ideals of $R$. It is known that $N(R)$ is equal to the lower nil-radical ${\rm Nil}_{\ast}(R)$ of $R$. A ring $R$ is called $2$-primal if $N(R)$ is the same as the set (nil-radical) ${\rm Nil}(R)$.

For an endomorphism $\alpha$ of $R$, the ring R is called $\alpha$-compatible if, for any elements $a$ and $b$ in $R$, the equality $ab = 0$ holds if, and only if, $a\alpha(b) = 0$. This definition is given in \cite{hashemi2005polynomial}. In this case, the map $\alpha$ must be injective.

We, thus, have all the ingredients necessary to prove the following assertion.

\begin{proposition}
Let $n$ be either odd, or $n=2^k$, or $6$, or $10$, and $R$ be a $2$-primal ring with $\alpha$ an endomorphism of $R$ such that $R$ is $\alpha$-compatible. The following are equivalent:
\begin{enumerate}
    \item $R[t, \alpha]$ is an $n$-UU ring.
    \item $R$ is an $n$-UU ring.
    \item $J(R) = {\rm Nil}(R)$ and $U^n(R/J(R))={\bar{1}}$.
\end{enumerate}
\end{proposition}

\begin{proof}
(i) $\Rightarrow$ (ii).  It is obvious.\\

(ii) $\Rightarrow$ (iii). Since $R$ is a $2$-primal ring, we deduce ${\rm Nil}(R)={\rm Nil}_{\ast}(R) \subseteq J(R)$. On the other hand, since $R$ is an $n$-UU ring, one must be that $J(R) \subseteq {\rm Nil}(R)$. Thus, ${\rm Nil}(R)= J(R)$, whence the quotient $R/J(R)$ is reduced, which gives $U^n(R/J(R))={\bar{1}}$.\\

(iii) $\Rightarrow$ (i). Since $J(R)={\rm Nil}(R)$, the factor $R/J(R)$ is reduced and we get $\alpha(J(R)) \subseteq J(R)$. However, $\overline{\alpha}:R/J(R) \to R/J(R)$, defined by $\overline{\alpha}(a+J(R))=\alpha(a)+J(R)$, is an endomorphism of $R/J(R)$. Moreover, by \cite[Lemma 2.4]{ouyang2011weak}, the ring $R/J(R)$ is $\alpha$-compatible. So, by \cite[Corollary 2.12]{chen2015constant}, we derive $$U(R/J(R)[x, \alpha])=U(R/J(R)).$$ Thus, $R/J(R)[x, \alpha]$ is an $n$-UU ring.

Furthermore, as $$R/J(R)[x, \alpha] \cong R[x, \alpha]/J(R)[x, \alpha],$$ it follows that $R[x, \alpha]/ J(R)[x, \alpha]$ is an $n$-UU ring. But, utilizing \cite[Lemma 2.2]{chen2015constant}, we know that ${\rm Nil}_{\ast}(R[x, \alpha])={\rm Nil}_{\ast}(R)[x, \alpha]$. So, the ideal $$J(R)[x, \alpha] = {\rm Nil}_{\ast}(R)[x, \alpha] = {\rm Nil}_{\ast}(R[x, \alpha])$$ is nil, as pursued.
\end{proof}

As a valuable consequence, we extract:

\begin{corollary}
Let $n$ be either odd, or $n=2^k$, or $6$, or $10$. Then, a $2$-primal ring $R$ is an $n$-UU ring if, and only if, $R[t]$ is an $n$-UU ring, if and only if, $J(R) = {\rm Nil}(R)$ and $U^n(R/J(R))={\bar{1}}$.
\end{corollary}

We now need the following claim.

\begin{lemma}
Let $D$ be a division ring such that $|D| > 2$. Suppose that $n \in \mathbb{N}$ is a constant, and $f(x)=1-x^n \in \mathbb{Z}[x]$. If there exists $d \in D \backslash \left\{ 0 \right\}$ such that $f(d)\neq 0$, then ${\rm M}_m(D)$ is not an $n$-UU ring for any $m\geq 2$.
\end{lemma}

\begin{proof}
Suppose there exists $d \in D \backslash \left\{ 0 \right\}$ such that $f(d)\neq 0$. Writing
$$A :=
\begin{pmatrix}
 f(d)& 0 \\
 0& 0
\end{pmatrix},$$ then one inspects that
$$I-A=
 \begin{pmatrix}
 d& 0 \\
 0& 1
\end{pmatrix}^n,$$
where $$\begin{pmatrix}
 d& 0 \\
 0& 1
\end{pmatrix} \in U(R).$$ So, this manifestly implies that the ring ${\rm M}_m(D)$ is {\it not} an $n$-UU ring, as asserted.
\end{proof}

We now establish the following two reduction statements.

\begin{lemma}
If $R$ is a ring and $m, n \in \mathbb{N}$ such that $(m, n) = d$, then $R$ is a $d$-UU ring, provided $R$ is simultaneously $m$-UU and $n$-UU.
\end{lemma}

\begin{proof}
Given $a \in U(R)$, there exist $q, p \in {\rm Nil}(R)$ such that $a^n = 1 + q$ and $a^m = 1 + p$. Since $(m, n) = d$, there exist $k, k' \in \mathbb{Z}$ such that $km + k'n = d$. Exploiting Proposition \ref{pro 1}, there exist $q', p' \in {\rm Nil}(R)$ such that $a^{k'n} = 1 + q'$ and $a^{km} = 1 + p'$. Consequently, $$a^{-km}(1-a^d) = a^{-km} - a^{k'n} \in {\rm Nil}(R).$$ Since $a \in U(R)$, we have $1-a^d \in {\rm Nil}(R)$, as required.
\end{proof}

As an automatic consequence, we obtain:

\begin{corollary}
If $R$ is a ring and $m, n \in \mathbb{N}$ such that $(m, n) = 1$, then $R$ is an UU ring, provided $R$ is both $m$-UU and $n$-UU.
\end{corollary}

\subsection{Group rings}

Firstly, we would like to recollect some background material on group rings that we will use in the sequel. In fact, if $R$ is a ring and $G$ is a group, the symbol $RG$ denotes the group ring of the group $G$ over $R$. The ring homomorphism $\omega : RG \to R$, defined by $$\sum r_gg \to \sum r_g,$$ is said to be {\it the augmentation map}, and $ker(\omega)$ is called {\it the augmentation ideal} of the group ring $RG$ and is denoted by $\Delta(RG)$. That is, $$\Delta(RG) = \{\sum r_gg \in RG: \sum r_g = 0\}.$$

\medskip

Next, let $\pi$ be a non-empty set of prime numbers. We shall say that an integer $n$ is {\it a $\pi$-number} if all its prime divisors lie within the set $\pi$. Let $G$ be a group and let $g\in G$ be an element of finite order, say ${\rm o}(g)$. Then, the element $g$ is called {\it a $\pi$-element} if ${\rm o}(g)$ is a $\pi$-number. Likewise, $G$ is said to be {\it a $\pi$-torsion group} if every element of $G$ is a $\pi$-element. Specifically, if $\pi = \{\ p\}$, then $G$ is the standard $p$-group. Therefore, in a $p$-group, the order of every element is a power of the prime number $p$. If, however, $G$ does not possess $\pi$-torsion elements other than the identity, then $G$ is known as a torsion-free $\pi$-group. Moreover, $G$ is termed {\it locally finite} if every finitely generated subgroup of $G$ is finite. It is clear that every finite group is locally finite, but the converse fails.

\medskip

Our purpose here is to expand \cite[Theorem 2.1]{DA} pertaining to the criterion of ({\it not} necessarily commutative) UU group rings over locally finite groups. Before doing that, we need the following pivotal statements.

\begin{theorem}\cite[Theorem 3.2]{danchev2016rings} \label{unipo thm}
Let $R$ be a ring with ${\rm char} (R) = m < \infty$. If $a \in R$ is such that $a^{s+1} = 0$ for some integer $s \ge 0$, then $(1 - a)^{m^s} = 1$.
\end{theorem}

We have now accumulated all the information necessary to prove the following major for this subsection result.

\begin{theorem}\label{thm 1 group ring}
Let $R$ be a ring with ${\rm char} (R) = m$ and let $G$ be a group. Also, for positive integers $m > 1$, $n \ge 1$, assume $\pi$ is the set of prime divisors of $m$ and $n$. If $RG$ is an $n$-UU ring, then $R$ is an $n$-UU ring and $G$ is a $\pi$-torsion group.
\end{theorem}

\begin{proof}
Since $R \subseteq RG$ may be viewed as a natural embedding, in view of Proposition \ref{sub}, $R$ is an $n$-UU ring. Now, if $g \in G$, then $g \in U(RG)$, and since $RG$ is an $n$-UU ring we have $g^n=1-q$, where $q \in {\rm Nil}(RG)$. Assuming that $q^{s+1}=0$ for some $s\in \mathbb{N}$, then Theorem \ref{unipo thm} gives that $(1-q)^{m^s}=1$, whence
$$g^{nm^s}=(g^n)^{m^s}=(1-q)^{m^s}=1,$$
as required.
\end{proof}

\begin{proposition}\label{pro 1 group ring}
Let $R$ be an $n$-UU ring such that $p \in {\rm Nil}(R)$, where $p$ is a prime number. Also, suppose that $G$ is a locally finite $p$-group. Then, $RG$ is an $n$-UU ring.
\end{proposition}

\begin{proof}
According to \cite[Proposition 16]{Connell}, the ideal $\Delta(RG)$ is nil. Furthermore, by combining the isomorphism $R \cong RG/\Delta(RG)$ and Lemma \ref{nil ideal}, the assertion is concluded.
\end{proof}

As three useful consequences, we yield:

\begin{corollary}
Let $R$ be an $n$-UU ring, where $n$ is odd, and also suppose that $G$ is a locally finite $2$-group. Then, $RG$ is an $n$-UU ring.
\end{corollary}

\begin{corollary}\cite[Theorem 2.1]{DA}
Let $G$ be a group and $R$ a ring.
\begin{enumerate}
    \item If $RG$ is a UU ring, then $R$ is a UU ring and $G$ is a $2$-group.
    \item If $G$ is locally finite, then $RG$ is a UU ring if, and only if, $R$ is a UU ring and $G$ is a $2$-group.
\end{enumerate}
\end{corollary}

\begin{proof}
Employing \cite[Theorem 2.6(1)]{danchev2016rings}, one knows that $2 \in {\rm Nil}(R)$ and ${\rm char}(R)=2^t$ where $t \ge 0$. Therefore, by setting $n=1$ in Theorem \ref{thm 1 group ring} and Proposition \ref{pro 1 group ring}, the proof is completed.
\end{proof}

\begin{lemma}
Let $R$ be a ring with ${\rm char}(R)=p^k$, where $p$ is a prime number, and let $G$ be a torsion-free $\pi$-group, where $\pi$ is the set of prime divisors of $n$. If $RG$ is an $n$-UU ring, then $R$ is an $n$-UU ring and $G$ is a $p$-group.
\end{lemma}

\begin{proof}
According to Theorem \ref{thm 1 group ring}, $R$ is an $n$-UU ring and $G$ is a $(\{p\} \cup \pi)$-torsion group, where $\pi$ is the set of prime divisors of $n$.

Assume $\pi = \{q_1, q_2, \ldots , q_k\}$, where, for every $1 \le i \le k$, $q_i$ is a prime number. Supposing $g \in G$, then $${\rm o}(g)=q_{1}^{\alpha_1} \cdots q_{k}^{\alpha_k}p^{\beta}.$$ However, $${\rm o}(g^{p^{\beta}})=q_{1}^{\alpha_1}...q_{k}^{\alpha_k}.$$ Consequently, the element $g^{p^{\beta}}$ is a $\pi$-element. And since $G$ is a torsion-free $\pi$-group, it must be that $g^{p^{\beta}}=1$, that is, $G$ is a $p$-group, as stated.
\end{proof}

As a consequence, we extract the following.

\begin{corollary}
Let $G$ be a torsion-free $\pi$-group, where $\pi$ is the set of prime divisors of $n$ and $n$ is odd. If $RG$ is an $n$-UU ring, then $R$ is an $n$-UU ring and $G$ is a $2$-group.
\end{corollary}

We close our examination with the following challenging question.

\begin{problem}\label{important} Find a necessary and sufficient condition for a group ring to be $n$-UU for an arbitrary $n\geq 2$.
\end{problem}

\medskip
\medskip
\medskip
	
\noindent{\bf Funding:} The first-named author, P.V. Danchev, of this research paper was partially supported by the Junta de Andaluc\'ia under Grant FQM 264, and by the BIDEB 2221 of T\"UB\'ITAK.

\end{document}